\documentclass[draft]{amsart}
\usepackage{proof}
\usepackage{amsmath, amssymb}
\usepackage{amsrefs}

\newtheorem{proposition}{Proposition}
\newtheorem{lemma}{Lemma}
\newtheorem{definition}{Definition}
\newtheorem{theorem}{Theorem}

\begin{document}

\title{On use of an explicit congruence predicate in Bounded Arithmetic}
\author{Yoriyuki Yamagata}
\address{National Institute of Advanced Science and Technology,
Mitsui-Sumitomo Kaijo Senri Bldg. 5F,
1-2-14 Sinsenri-Nishimachi,
560-0083 Toyonaka,
Japan}
\email{yoriyuki.yamagata@aist.go.jp}

\begin{abstract}
 We introduce system $S^2_0E$, a bounded arithmetic corresponding
 to Buss's $S^2_0$ with the predicate $E$ which signifies the existence
 of the value.  Then, we show that we can $\Sigma^b_2$-define truthness
 of $S^2_0 E$ and therefore we can prove consistency of $S^2_0 E$ in
 $S^2_2$.  Finally, we conjecture that $S^2_0 E + \Sigma^b_1-PIND$
 interprets $S^2_1$.
\end{abstract}

\maketitle

\section{Introduction}

One of the central questions concerning Bounded Arithmetic is
whether Buss's hierarchy $S^2_1 \subseteq S^2_2 \subseteq \dots$ of
theories collapses \cite{Buss:Book}.  Natural way to show difference
between these theories is to look whether these theories proves (some
appropriate formulation of ) a consistency statement of some theory $T$.
However, known results are mostly negative.  Pudl\'ak
\cite{PudlakNote} shows $S^2$ cannot prove bounded consistency of
$S^2_1$.  Buss and Ignjatovi\'c \cite{BussUnprovability} improve Pudl\'ak
result showing that $S^2_i$ cannot prove $B_i$-bounded consistency of
$S^2_0$.  

Here, natural question arises: is there a ``sufficiently strong'' theory
which can be proved consistent inside $S^2_i$ for some $i \in
\mathbf{N}$?  By ``sufficiently strong'' theory we mean a theory $T$
which can be used as a replacement of $S^2_0$ to formalize $S^2_1$.  In
other word, we mean a theory $T$ such that $T$ plus $\Sigma^b_1$-$PIND$
can interpret $S^2_1$.

To prove consistency of $T$ inside $S^2_i$, natural way is having a
truth definition of the language of $T$ inside $S^2_i$.  Looking Takeuti
\cite{Takeuti:Truth}, main difficulty to have a truth definition of the
language of $S^2$ inside $S^2_i$ is the fact that $S^2_i$ cannot
uniformly prove the existence of valuation of the terms.
This suggests that adding a predicate which signifies convergence of
terms to $S^2_i$ makes consistency proof of $S^2_i$ easier.

In this paper, we define $S^0_2 E$, a bounded arithmetic with an
explicit congruence predicate and prove its soundness inside $S^2_2$.
$S^0_2E$ does not have any induction axioms, hence it corresponds
$S^0_2$ in Buss's hierarchy.  In comparison to $S^0_2$, it is very weak
since for example, it does not contain commutativity of $+, \cdot$ and
so on.  However, we conjecture that if we add $S^0_2E$
$\Sigma^b_1$-PIND, it can interpret $S^2_1$.  This conjecture is
supported by the fact that $S^0_2E$ contains all inductive definition of
functions and predicates.  The language of $S^2_0 E$ is restricted to
$\Sigma^b_1$-formulas so that $S^2_2$ can prove its soundness.

This paper is organized as follows.  In Section \ref{sec:S02E}, The
system $S^0_2 E$ is introduced.  In Section \ref{sec:truthdef}, truth
definition of $S^0_2 E$ inside $S^2_2$ is given.  In Section
\ref{sec:soundness}, soundness and consistency of $S^0_2 E$ is proved
inside $S^2_2$.  Finally, in Section \ref{sec:conjectures} several
conjecture concerning $S^0_2E$ are given.

\section{A bounded arithmetic $S^0_2E$ using explicit congruence operator $E$}\label{sec:S02E}

\begin{definition}[Language of $S^0_2E$] Language of $S^0_2E$ consists
 of the following symbols.
  \begin{itemize}
   \item constant 0
   \item Variables $x_1, x_2, \dots$ (denoted $x, y, a, b$, etc. )
   \item unary function symbols $S, \lfloor \frac{}{2} \rfloor, |\ |,
	 s_0, s_1$ and binary function symbol$+, \times, \#$
   \item unary predicate symbols $E$ and binary predicate (relation)
	 symbol $\leq, =$
   \item logical symbols $\vee, \wedge, \neg$ and quantifiers $\forall,
	 \exists$. 
  \end{itemize}
\end{definition}

\begin{definition}[Terms of $S^0_2E$]
Terms of $S^0_2E$ are defined as recursively as follows.
\begin{itemize}
 \item Variables $x_1, x_2, x_3, \dots$ are terms.  We use metavariable
       $x, y, z$ to denote variables.
 \item Constant $0$ is term.
 \item If $t_1, t_2$ are terms, $St_1, t_1 + t_2, t_1 \times t_2,
       \lfloor \frac{t_1}{2} \rfloor, |t_1|, t_1 \# t_2, s_0 t_1, s_1
       t_1$ are terms.
\end{itemize}
 We use $t_1, t_2, \dots, t, s, u$ to denote terms.  We say a term $t$
 \emph{sharply bounded} if $t$ has a form $|t'|$.  For any natural
 number, there is a standard notation using shortest combination of $0,
 s_0, s_1$.  If we use numerals $1, 2, 3, \dots$ in the language of
 $S^0_2E$, we understand that they are represented by such a standard
 notation.
\end{definition}

\begin{definition}
 Formulas of $S^0_2E$ are defined as follows.
 \begin{itemize}
  \item For terms $t_1, \dots, t_n$ and $n$-ary predicate symbol $p$,
	$pt_1 \dots t_n$ and $\neg p t_1 \dots t_n$ are formulas.  We
	often use $t \not= u$ and $t \not\leq u$ to denote $\neg t = u$
	and $\neg t \leq u$ respectively.
  \item If $\phi$ and $\psi$ are formulas, $\phi \vee \psi$ and $\phi
	\wedge \psi$ are formulas.
  \item If $\phi$ is a formula , $t \equiv |u|$ is a sharply bounded term and
	$x$ is a variable, the form $\forall x \leq t \phi$ is a
	formula.  We say quantifier in the form $\forall x \leq |t|$
	\emph{sharply bounded}.
  \item  If $\phi$ is a formula, $t$ is a term and $x$ is a variable,
	 the form $\exists x \leq t \phi$ is a formula.
 \end{itemize}
 We call a formula in the form $pt_1 \dots t_n$ ($p$ : predicate, $t_1,
 \dots, t_n$ : terms) \emph{atomic}
\end{definition}

\begin{definition}\label{defn:axiom}
 Axioms of $S^0_E$ are sequents defined as follows.
 \begin{description}
  \item[E-axioms] 
	     \begin{equation}
	      \rightarrow E0
	     \end{equation}
	     \begin{equation}
	      Ex \rightarrow Es_ix 
	     \end{equation} where $i = 0 \text{ or } 1$.
	     \begin{equation}
	      p t_1 \dots t_n \rightarrow Et_i
	     \end{equation} where $i = 1 \dots n$.
	     \begin{equation}
	      \neg p t_1 \dots t_n \rightarrow Et_i
	     \end{equation} where $i = 1 \dots n$.
  \item[Equality axioms] 
	     \begin{equation}
	      Ex \rightarrow x=x
	     \end{equation}
	     \begin{equation}
	      x=y, y=z \rightarrow x=z
	     \end{equation}
	     \begin{equation}
	      x=y \rightarrow s_i x = s_i y
	     \end{equation} where $i = 0 \text{ or } 1$.
  \item[Separation axioms] 
	     \begin{equation}
	      x \not= 0 \rightarrow x \not= s_0x
	     \end{equation}
	     \begin{equation}
	      Ex \rightarrow x \not= s_1x
	     \end{equation}
	     \begin{equation}
	      Ex \rightarrow s_0 x \not= s_1x
	     \end{equation}
  \item[Inequality axioms] 
	     \begin{equation}
	      Ex \rightarrow 0 \leq x
	     \end{equation}
	     \begin{equation}
	      x \leq y \rightarrow s_i x \leq s_i y
	     \end{equation} where $i = 0 \text{ or } 1$
	     \begin{equation}
	      x \leq y \rightarrow s_0 x \leq s_1 y
	     \end{equation}
  \item[Defining axioms] 
	     \begin{description}
	      \item[$Cond$] 
			 \begin{align}
			  Ey, Ez &\rightarrow Cond (0, y, z) = y \\
			  ECond(x, y, z) &\rightarrow Cond (s_0 x, y, z)
			   = Cond (x, y, z) \\
			  Ex, Ey, Ez &\rightarrow Cond (s_1 x, y, z) = z
			 \end{align}
	      \item[$S$] 
			 \begin{align}
			  &\rightarrow S0 = s_10 \\
			  Es_1x &\rightarrow Ss_0x = s_1x \\
			  ESx &\rightarrow Ss_1x = s_0(Sx)
			 \end{align}
	      \item[$|\ |$] 
			 \begin{align}
			  &\rightarrow |0|=0 \\
			  ES|x| &\rightarrow |s_0x| = Cond (x, 0,
			  S|x|)\\
			  ES|x| &\rightarrow |s_1x| = S|x|
			 \end{align}
	      \item[$\lfloor \frac{}{2} \rfloor$] 
			 \begin{align}
			  &\rightarrow \lfloor \frac{0}{2} \rfloor = 0
			  \\
			  Ex &\rightarrow \lfloor \frac{1}{2} s_0x
			  \rfloor = x \\
			  Ex &\rightarrow \lfloor \frac{1}{2} s_1x
			  \rfloor = x
			 \end{align}
	      \item[$\boxplus$] 
			 \begin{align}
			  Ex &\rightarrow x \boxplus 0 = x \\
			  Es_0(x \boxplus y) &\rightarrow x \boxplus s_0
			  y = Cond (y, x, s_0 (x \boxplus y))\\
			  Es_0(x \boxplus y) &\rightarrow x \boxplus s_1
			  y = s_0 (x \boxplus y)
			 \end{align}
	      \item[$\#$] 
			 \begin{align}
			  Ex &\rightarrow x \# 0 = 1\\
			  E (x \# y) \boxplus x &\rightarrow x \# s_0 y
			  = Cond (y, 1, (x \# y) \boxplus x)\\
			  E (x \# y) \boxplus x &\rightarrow x \# s_1y =
			  (x \# y) \boxplus x 
			 \end{align}
	      \item[$parity$] 
			 \begin{align}
			  &\rightarrow parity(0) = 0\\
			  Ex &\rightarrow parity(s_0 x) = 0\\
			  Ex &\rightarrow parity(s_1 x) = 1
			 \end{align}
	      \item[$+$] 
			 \begin{align}
			  Ex &\rightarrow x + 0 = x \\
			  E(\lfloor \frac{1}{2} x \rfloor + y)
			  &\rightarrow x + s_0 y = Cond(parity(x),
			  s_0(\lfloor \frac{1}{2} x \rfloor + y) ,
			  s_1(\lfloor \frac{1}{2} x \rfloor + y))\\
			  E(\lfloor \frac{1}{2} x \rfloor + y)
			  &\rightarrow x + s_1 y = Cond(parity(x),
			  s_1(\lfloor \frac{1}{2} x \rfloor + y) ,
			  s_0(S(\lfloor \frac{1}{2} x \rfloor + y)))
			 \end{align}
	      \item[$\cdot$] 
			 \begin{align}
			  Ex &\rightarrow x \cdot 0 = 0\\
			  E x \cdot y &\rightarrow x \cdot (s_0 y) =
			  s_0(x \cdot y)\\
			  Es_0(x \cdot y) + x &\rightarrow x \cdot (s_1
			  y) = s_0(x \cdot y) + x
			 \end{align}
	     \end{description}
 \end{description}
\end{definition}

\begin{definition}
 Let $\Gamma(\vec{x}) \rightarrow \Delta(\vec{x})$ be a sequent with
 free variables $\vec{x}$.  Then, \emph{substitution instance}
 $\Gamma(\vec{t}) \rightarrow \Delta(\vec{t})$ of $\Gamma(\vec{x})
 \rightarrow \Delta(\vec{x})$ is a sequent obtained by substituting
 terms $\vec{t}$ to free variables $\vec{x}$ in $\Gamma(\vec{x})
 \rightarrow \Delta(\vec{x})$.
\end{definition}

\begin{definition}
 The inference rules of $S^0_2E$ are defined as follows.
 \begin{description}
  \item[Identity rule] 
	     $$\infer{a \rightarrow a}{}$$
	     where $a$ is an atomic formula.
  \item[Axioms] 
	     $$\infer{\Gamma \rightarrow \Delta}{}$$
	     if $\Gamma \rightarrow \Delta$ is an substitution instance
	     of axioms defined in Definition \ref{defn:axiom}.
  \item[Stractural rules] 
	     \begin{description}
	      \item[Weakening rule] 
			 $$\infer{A, \Gamma \rightarrow
			 \Delta}{\Gamma \rightarrow \Delta}$$
			 $$\infer{\Gamma \rightarrow
			 \Delta, A}{\Gamma \rightarrow \Delta}$$
	      \item[Contraction] 
			 $$\infer{A, \Gamma \rightarrow
			 \Delta}{A, A, \Gamma \rightarrow \Delta}$$
			 $$\infer{\Gamma \rightarrow
			 \Delta, A}{\Gamma \rightarrow \Delta, A, A}$$
	      \item[Exchange] 
			 $$\infer{\Gamma, B, A, \Pi \rightarrow
			 \Delta}{\Gamma, A, B, \Pi \rightarrow \Delta}$$
			 $$\infer{\Gamma \rightarrow \Delta, B, A,
			 \Pi}{\Gamma \rightarrow \Delta, A, B, \Pi}$$
	     \end{description}
  \item[Logical rules]
	     \begin{description}
	      \item[$\neg$-rules] $$\infer{\neg p(t_1, \dots, t_n),
			 \Gamma \rightarrow \Delta}{\Gamma \rightarrow
			 \Delta, p(t_1, \dots, t_n)}$$ where $p$ is a
			 $n$-ary predicate.  $$\infer{Et_1, \dots, Et_n,
			 \Gamma \rightarrow \Delta, \neg p(t_1, \dots,
			 t_n)}{p(t_1, \dots, t_n), \Gamma \rightarrow
			 \Delta}$$
	      \item[$\wedge$-rules] $$\infer{A \wedge B. \Gamma
			 \rightarrow \Delta}{A, \Gamma \rightarrow
			 \Delta}$$ 
			 $$\infer{B \wedge A. \Gamma
			 \rightarrow \Delta}{A, \Gamma \rightarrow
			 \Delta}$$ 
			 $$\infer{\Gamma \rightarrow \Delta, A \wedge
			 B}{\Gamma \rightarrow \Delta, A \quad \Gamma
			 \rightarrow \Delta, B}$$
	      \item[$\vee$-rules] $$\infer{A \vee B, \Gamma \rightarrow
			 \Delta}{A, \Gamma \rightarrow \Delta \quad B,
			 \Gamma \rightarrow \Delta}$$
			 $$\infer{\Gamma \rightarrow \Delta, A \vee
			 B}{\Gamma \rightarrow \Delta, A}$$
			 $$\infer{\Gamma \rightarrow \Delta, B \vee
			 A}{\Gamma \rightarrow \Delta, A}$$
	      \item[$\forall$-rules] $$\infer{t \leq s, \forall x \leq s.
			 A(x), \Gamma \rightarrow \Delta}{A(t), \Gamma
			 \rightarrow \Delta}$$
			 $$\infer{Nt, \Gamma \rightarrow \Delta, \forall x
			 \leq t. A(x)}{x \leq t, \Gamma \rightarrow
			 \Delta, A(x)}$$ where $x$ does not appear in
			 $\Gamma, \Delta$ and $t$.
	      \item[$\exists$-rules] $$\infer{\exists x \leq t. A(x),
			 \Gamma \rightarrow \Delta}{x \leq t, A(x),
			 \Gamma \rightarrow \Delta}$$ where $x$ does not
			 appear in $\Gamma, \Delta$.
			 $$\infer{t \leq s, \Gamma \rightarrow \Delta,
			 \exists x \leq s. A(x)}{\Gamma \rightarrow
			 \Delta, A(t)}$$
	     \end{description}
  \item[Cut-rule] $$\infer{\Gamma, \Pi \rightarrow \Delta,
	     \Lambda}{\Gamma \rightarrow \Delta, A \quad
	     A, \Pi \rightarrow \Lambda}$$
 \end{description}
\end{definition}

\section{Truth definition of $S^0_2E$}\label{sec:truthdef}

\begin{definition}
 A tree $w$ is $\vec{b}$-valuation tree of a term $t(\vec{a})$ bounded by $u$ if and only if 
\begin{enumerate}
 \item All nodes of $w$ has a form $\langle \lceil t_0 \rceil, c
       \rangle$ where $\lceil t_0 \rceil$ is a G\"odel number of a
       subterm $t_0$ of $t$ and $c \leq u$.
 \item Leafs of $w$ are either in the form $\langle \lceil 0 \rceil, 0
       \rangle$ or $\langle \lceil a_i \rceil, b_i \rangle$.
 \item The root of $w$ has a form $\langle \lceil t(\vec{a}) \rceil, c
       \rangle$. 
 \item Child nodes of a node $\langle \lceil f(t_1, \dots, t_n) \rceil,
       c \rangle$ are $\langle \lceil t_1 \rceil, d_1 \rangle, \dots,
       \langle \lceil t_n \rceil, d_n \rangle$ and $c = f(d_1, \dots,
       d_n)$ holds.
\end{enumerate}

If the roof of $\vec{b}$-valuation tree $w$ has a form $\langle \lceil t
 \rceil, c \rangle$, we say \emph{the value of $w$ is $c$}.

We define $v(\lceil t \rceil, \vec{b}) \downarrow_u c
 \Leftrightarrow_{def} \exists w \leq s(\lceil t \rceil, u)$ ''w is
 a $\vec b$-valuation tree bounded by $u$ and the root of
 $w$ is $\langle \lceil t \rceil, c \rangle$'' where $s$ is a suitable
 term to bound the size of valuation trees of $t$ bounded by $u$.  Then,
 $v(\lceil t \rceil, \vec{b}) \downarrow_u c$ is a $\Sigma^b_1$-formula.
\end{definition}

\begin{lemma}\label{lem:v1}  The following statements are provable in $S^2_1$.
 \begin{enumerate}
  \item If $w$ is a $\vec{b}$-valuation tree bounded by $u$ and $u \leq
	u'$, $w$ is a $\vec{b}$-valuation tree bounded by $u'$.
  \item $v(\lceil t \rceil, \rho) \downarrow_u c$ and $u < u'$, then
	$v(\lceil t \rceil, \rho) \downarrow_u' c$.
  \item $v(\lceil t \rceil, \rho) \downarrow_u c$ and $v(\lceil t
	\rceil, \rho) \downarrow_u c'$, then $c = c'$.
 \end{enumerate}
\end{lemma}

\begin{lemma}\label{lem:v2} The following statements are provable in $S^2_1$.
 \begin{enumerate}
  \item $v(\lceil f(t_1(\vec(a)), \dots, t_k(\vec(a)))
	\rceil, \vec{b}) \downarrow_u c$ then, $\exists d_1, \dots, d_k,
	v(\lceil t_1(\vec(a)) \rceil, \vec{b}) \downarrow_u d_1,$ $\ldots$
	$v(\lceil t_k(\vec(a)) \rceil, \vec{b}) \downarrow_u d_k$ and
	$f(d_1, \dots, d_k) = c$.
  \item $v(\lceil 0 \rceil, \vec{b}) \downarrow_0 0$
  \item $v(\lceil a_j \rceil, \vec{b}) \downarrow_{b_j} b_j$
  \item $v(\lceil t(\vec{a}, t'(\vec{a})) \rceil, \vec{b}) \downarrow_u c
	\leftrightarrow  \exists c' \leq u, v(\lceil t'(\vec{a}) \rceil,
	\vec{b})
	\downarrow_u c' \wedge v(\lceil t(\vec{a}, a) \rceil, \vec{b} *
	c') \downarrow_u c $
 \end{enumerate}
\end{lemma}

\begin{definition}\label{defn:T_0}
 Assume that $\phi$ is a quantifier free formula of $S^0_2E$.  We define
 \emph{$\vec{b}$-truth tree bounded by $u$ of $\phi(\vec{a})$} as a tree $w$ satisfying
 the following condition.
\begin{enumerate}
 \item All nodes of $w$ has a form $\langle \lceil \psi \rceil,
       \epsilon \rangle$.
 \item The root of $w$ has a form $\langle \lceil \phi(\vec{a}) \rceil,
       \epsilon \rangle$.
 \item The leaf of $w$ has a form $\langle \lceil t_1 \leq t_2 \rceil,
       \epsilon \rangle$ or $\langle \lceil t_1 \not\leq t_2 \rceil,
       \epsilon \rangle$ or $\langle \lceil t_1 = t_2 \rceil, \epsilon
       \rangle$ or $\langle \lceil t_1 \not= t_2 \rceil, \epsilon
       \rangle$ or $\langle \lceil Et \rceil, \epsilon \rangle$.  For
       the case of that the leaf has a form $\langle \lceil t_1 \leq t_2
       \rceil, \epsilon \rangle$, $\epsilon = 1$ if and only if $\exists c_1, c_2
       \leq u$, $v(\lceil t_1 \rceil, \vec{b}) \downarrow_u c_1$,
       $v(\lceil t_2 \rceil, \vec{b}) \downarrow_u c_2$ and $c_1 \leq
       c_2$.  Otherwise $\epsilon = 0$.  For the case of that the leaf
       has a form $\langle \lceil t_1 \not\leq t_2 \rceil, \epsilon
       \rangle$, $\epsilon = 1$ if and only if $\exists c_1, c_2 \leq u$, $v(\lceil
       t_1 \rceil, \vec{b}) \downarrow_u c_1$, $v(\lceil t_2 \rceil,
       \vec{b}) \downarrow_u c_2$ and $c_1 \not\leq c_2$.  Otherwise
       $\epsilon = 0$.  For the cases of $\langle \lceil t_1 = t_2
       \rceil, \epsilon \rangle$ or $\langle \lceil t_1 \not= t_2
       \rceil, \epsilon \rangle$, the conditions are similar.  For the
       case that the leaf is $\langle \lceil E t \rceil, \epsilon
       \rangle$, $\epsilon = 1$ if and only if $\exists c \leq u$, $v(\lceil t
       \rceil, \vec{b}) \downarrow_u c$.  Otherwise $\epsilon = 0$.
 \item If the node $r$ of $w$ has a form $\langle \lceil \phi_1 \wedge
       \phi_2 \rceil, \epsilon \rangle$, $r$ has child nodes $\langle
       \lceil \phi_1 \rceil, \epsilon_1 \rangle$, $\langle \lceil \phi_2
       \rceil, \epsilon_2 \rangle$ and $\epsilon = 1$ if and only if $\epsilon_1 =
       1$ and $\epsilon_2 = 1$.  Otherwise $\epsilon = 0$.
 \item If the node $r$ of $w$ has a form $\langle \lceil \phi_1 \vee
       \phi_2 \rceil, \epsilon \rangle$, $r$ has child nodes $\langle
       \lceil \phi_1 \rceil, \epsilon_1 \rangle$, $\langle \lceil \phi_2
       \rceil, \epsilon_2 \rangle$ and $\epsilon = 1$ if and only if $\epsilon_1 =
       1$ or $\epsilon_2 = 1$.  Otherwise $\epsilon = 0$.
\end{enumerate}
 We define $T_0(u, \lceil \phi(\vec{a}) \rceil, \vec{b})
 \Leftrightarrow_{def} \exists w \leq s(\lceil \phi(\vec{a}) \rceil,
 u),$ ``$w$ is a $\vec{b}$-truth tree bounded by $u$ and the root of $w$
 has a form $\langle \lceil \phi(\vec{a}) \rceil, 1 \rangle$.'' where
 $s$ is a suitable term bound the size of $w$ by $t$ and $u$.  Then,
 $T_0(u, \lceil \phi(\vec{a}) \rceil, \vec{b})$ is a
 $\Sigma^b_1$-formula. 
\end{definition}

\begin{lemma}\label{lem:T_0}
 $S^1_2$ proves the following statements.
\begin{enumerate}
 \item $T_0(u, \lceil t_1 \leq t_2 \rceil, \vec{b}) \leftrightarrow \exists
       c_1, c_2 \leq u, v(\lceil t_1 \rceil, \vec{b}) \downarrow_u c_1 \wedge
       v(\lceil t_2 \rceil, \vec{b}) \downarrow_u c_2 \wedge c_1 \leq c_2$
 \item $T_0(u, \lceil t_1 \not\leq t_2 \rceil, \vec{b}) \leftrightarrow \exists
       c_1, c_2 \leq u, v(\lceil t_1 \rceil, \vec{b}) \downarrow_u c_1 \wedge
       v(\lceil t_2 \rceil, \vec{b}) \downarrow_u c_2 \wedge c_1
       \not\leq c_2$
 \item $T_0(u, \lceil t_1 = t_2 \rceil, \vec{b}) \leftrightarrow \exists
       c_1, c_2 \leq u, v(\lceil t_1 \rceil, \vec{b}) \downarrow_u c_1 \wedge
       v(\lceil t_2 \rceil, \vec{b}) \downarrow_u c_2 \wedge c_1
       = c_2$
 \item $T_0(u, \lceil t_1 \not= t_2 \rceil, \vec{b}) \leftrightarrow \exists
       c_1, c_2 \leq u, v(\lceil t_1 \rceil, \vec{b}) \downarrow_u c_1 \wedge
       v(\lceil t_2 \rceil, \vec{b}) \downarrow_u c_2 \wedge c_1
       \not= c_2$  
 \item $T_0(u, \lceil Et \rceil, \vec{b}) \leftrightarrow \exists
       c \leq u, v(\lceil t \rceil, \vec{b}) \downarrow_u c$
 \item $T_0(u, \lceil \phi_1 \wedge \phi_2 \rceil, \vec{b})
       \leftrightarrow T_0(u, \lceil \phi_1 \rceil, \vec{b}) \wedge
       T_0(u, \lceil \phi_2 \rceil, \vec{b})$
 \item $T_0(u, \lceil \phi_1 \vee \phi_2 \rceil, \vec{b})
       \leftrightarrow T_0(u, \lceil \phi_1 \rceil, \vec{b}) \vee
       T_0(u, \lceil \phi_2 \rceil, \vec{b})$
 \item $T_0(u, \lceil \phi(\vec{a}, t(\vec{a})) \rceil, \vec{b})
       \leftrightarrow \exists c \leq u, v(\lceil t(\vec{a}) \rceil,
       \vec{b}) \downarrow_u c \wedge T_0(u,
       \lceil \phi(\vec{a}, a) \rceil, \vec{b} * c)$ 
 \item $T_0(u, \lceil \phi \rceil, \vec{b}), u \leq u' \rightarrow
       T_0(u', \lceil \phi \rceil, \vec{b})$
\end{enumerate}
\end{lemma}

\begin{proof}
 Only (8) is non-trivial.  (8) is proved by induction on the
 construction of $\phi$.  Using (6), (7), it is proved by considering
 the case where $\phi$ is atomic.  But if $\phi$ is atomic, the proof is
 handled by Clause (4) of Lemma \ref{lem:v2}.
\end{proof}

\begin{lemma}\label{lem:EM}
If $T_0(u, \lceil t_1 \rceil, \rho), \dots, T_0(u, \lceil t_1 \rceil,
\rho)$, then either $T(u, \lceil p(t_1, \cdots, t_n) \rceil, \rho)$ or
 $T(u, \lceil \neg p(t_1, \cdots, t_n) \rceil, \rho)$ holds.
\end{lemma}

\begin{definition}
 $\phi(\vec{a})$ is called \emph{pure 1-form} if and only if it has a
 form $$\exists x_1 \leq t_1(\vec{a}) \forall x_2 \leq |t_2(\vec{a},
 x_1)| A(\vec{a}, x_1, x_2)$$ where $A$ is quantifier-free and does not
 contain predicate $E$.
 
 Formula $\psi$ are called \emph{1-form} if it is subformula of a formula
 in pure 1-form, i.e.
\begin{itemize}
 \item $\psi$ is in pure 1-form, or
 \item $\psi$ has a form $\forall x \leq |t(\vec{a})| A(\vec{a}, x)$ where
       $A$ is quantifier-free and does not contain predicate $E$, or
 \item $\psi$ is quantifier-free and does not contain predicate $E$.
\end{itemize}
\end{definition}

\begin{definition}\label{defn:T}
 $T(u, \lceil \phi(\vec{a}) \rceil, \vec{b})$ is defined as the formula
 stating ``$\lceil \phi(\vec{a}) \rceil$ is a G\"odel number of 1-form
 or formula in a form $Et$ and one of the following clauses holds.
\begin{itemize}
 \item $\phi(\vec{a})$ is in the form $\exists x_1 \leq t_1(\vec{a})
       \forall x_2 \leq |t_2(\vec{a}, x_1)| A(\vec{a}, x_1,
       x_2)$. Moreover $\exists c \leq u, v(\lceil t_1(\vec{a}) \rceil,
       \vec{b}) \downarrow_u c$ and $\exists x_1 \leq c, \exists d \leq
       u, v(\lceil t_2(\vec{a}, x_1) \rceil, \vec{b} * x_1) \downarrow_u
       d$ and $\forall x_2 \leq |d|, T_0(u, \lceil A(\vec{a}, x_1, x_2)
       \rceil, \vec{b}*x_1*x_2)$
       holds. 
 \item $\phi(\vec{a})$ is in the form $ \forall x \leq |t(\vec{a})|
       A(\vec{a}, x)$. Moreover $\exists c \leq u, v(\lceil t(\vec{a})
       \rceil, \vec{b}) \downarrow_u c$ and $\forall x \leq |c|, T_0(u,
       \lceil A(\vec{a}, x) \rceil, \vec{b}*x)$ holds.
 \item $\phi(\vec{a})$ is quantifier-free.  Then $T_0(u, \lceil
       \phi(\vec{a}) \rceil, \vec{b})$ holds.''
\end{itemize}
 $T(u, \lceil \phi(\vec{a}) \rceil, \vec{b})$ is $\Sigma^b_1$-formula.
\end{definition}

\begin{lemma}\label{lem:T}
 $S^1_2$ proves $T(u, \lceil \phi(\vec{a}) \rceil, \vec{b}) \wedge u
 \leq u' \rightarrow T(u', \lceil \phi(\vec{a}) \rceil, \vec{b})$ 
\end{lemma}

\section{Soundness and consistency proof of $S^0_2E$ in $S^2_2$}\label{sec:soundness}

\begin{definition}
 A proof $p$ of $S^0_2E$ is \emph{strictly 1-normal} if and only if
\begin{enumerate}
 \item All formulas of $p$ is either 1-form or has a form $Et$.
 \item $p$ is a free variable normal form.
\end{enumerate}
 The property ``$w$ is a G\"odel number of strictly 1-normal proof of
 sequent $\Gamma \rightarrow \Delta$'' is $\Delta^b_1$-definable
 property.  We write $1-sPrf(w, \lceil \Gamma \rightarrow \Delta
 \rceil)$ the formula representing this property.
\end{definition}

\begin{proposition}\label{prop:soundness}
 Assume $1-sPrf(w, \lceil \Gamma \rightarrow \Delta \rceil)$.  For each
 node $r$ of $w$, we write the sequent of this node $\Gamma_r
 \rightarrow \Delta_r$ and number of parameter variables in $\Gamma_r
 \rightarrow \Delta_r$ $k_r$.  Then, for each node $r$ of $w$ and any
 $u$, the following hold.

\begin{multline}\label{eq:soundness}
  \forall \rho \leq u \# 2^{k_r} (seq(\rho) \wedge Len(\rho) = k_r
 \wedge \forall i < k_r (\beta(i+1, \rho) \leq u)) \supset \\
 \forall u'
 \leq u \circleddash r \forall A \in \Gamma_r \ T(u', \lceil A \rceil, \rho)
\supset \exists B \in \Delta_r \ T(u' \oplus r, \lceil B \rceil,
\rho)\end{multline}

where $\circleddash$ is a bit-subtraction and $\oplus$ is a
 bit-concatenation.

Furthermore, this is derivable in $S^2_2$.
\end{proposition}

\begin{proof}
 Tree induction on $w$.  Note that the fo]rmula (\ref{eq:soundness}) is
 $\Sigma^b_2$-formula.  Hence, $S^2_2$ can formalize this induction.  We
 reason informally inside $S^2_2$.

 We distinguish different cases of the
 inference deriving $\Gamma_r \rightarrow \Delta_r$.

\begin{description}
 \item[Identity rule] 

	     $$\infer{a \rightarrow a}{}$$

	    Let $\rho \leq u \# 2^{k_r}$ and assume $Seq(\rho)$,
	    $Len(\rho)=k_r$, $\forall i < k_r \beta(i+1, \rho) \leq u$.
	    Let $u' \leq u \circleddash r$.  Further, assume that $T(u',
	    \lceil a \rceil, \rho)$.  Then, by Lemma \ref{lem:T}, $T(u'
	    \oplus r, \lceil a \rceil, \rho)$.  Hence, $r$ satisfies
	    (\ref{eq:soundness}).

 \item[Axioms] We distinguish different cases based on which axiom the
	    sequent is a substitution instance of.  Let $\rho \leq u \#
	    2^{k_r}$ and assume $Seq(\rho)$, $Len(\rho)=k_r$, $\forall i <
	    k_r \beta(i+1, \rho) \leq u$.  Let $u' \leq u \circleddash
	    r$.
	    \begin{description}
	     \item[E-axioms]

			$$\rightarrow E0$$

			Since $v(\lceil 0 \rceil, \rho)\downarrow_0 0$,
			$T(u' \oplus r, \lceil E0 \rceil, \rho)$.  Hence
			$r$ satisfies (\ref{eq:soundness}).

			$$Et \rightarrow Es_it$$ for $i = 0, 1$.  Assume
			$T(u', \lceil Et \rceil, \rho)$ i.e. $\exists c
			\leq u' v(\lceil t \rceil, \rho) \downarrow_{u'}
			c$.  Then $v(\lceil s_i t \rceil, \rho)
			\downarrow_{s_i u'} s_i c$.  Hence $T(s_i u',
			\lceil Es_i t \rceil, \rho)$.  Since $s_i u'
			\leq u' \oplus r$, $T(u' \oplus r, \lceil Es_i t
			\rceil, \rho)$.  Hence $r$ satisfies
			(\ref{eq:soundness}).

			$$p t_1 \dots t_n \rightarrow Et_i$$ for $i = 1,
			\dots, n$.  Assume $T(u', \lceil p t_1 \dots t_n
			\rceil, \rho)$.  By definition \ref{defn:T} of
			$T$, $T_0(u', \lceil p t_1 \dots t_n \rceil,
			\rho)$.  By definition \ref{defn:T_0} of $T_0$,
			there is $\rho$-truth tree $w$ bounded by $u$.
			Only node of $w$ consists of $\langle \lceil p
			t_1 \dots t_n \rceil, 1 \rangle$.  By definition
			it is the case only when $\exists c \leq u'
			v(\lceil t_i \rceil, \rho)\downarrow_{u'} c$.
			Hence $T(u', \lceil Et_i \rceil, \rho)$.  By
			Lemma \ref{lem:T}, $T(u' \oplus r, \lceil Et_i
			\rceil, \rho)$.  Hence $r$ satisfies
			(\ref{eq:soundness}).

			The case $$\neg p t_1 \dots t_n \rightarrow
			Et_i$$ is treated similarly as above.
	     \item[Equality axioms] $$Et \rightarrow t=t$$ Assume that
			$T(u', \lceil Et \rceil, \rho)$.  Then, $\exists
			c \leq u' v(\lceil t \rceil,
			\rho)\downarrow_{u'} c$.  
			Since $c=c$,  there is $\rho$-truth tree $w$
			bounded by $u$ which consists of single node
			$\langle \lceil t=t \rceil, 1 \rangle$.  Hence
			$T_0(u', \lceil t=t \rceil, \rho)$ and therefore
			$T(u', \lceil t=t \rceil, \rho)$.  By Lemma
			\ref{lem:T}, $T(u' \oplus r, \lceil t=t \rceil,
			\rho)$.  Hence $r$ satisfies
			(\ref{eq:soundness}).

			$$t_1=t_2, t_2=t_3 \rightarrow t_1=t_3$$ Assume
			that $T(u', \lceil t_1=t_2 \rceil, \rho)$ and
			$T(u', \lceil t_2=t_3 \rceil, \rho)$.  Then,
			$\exists c1 \leq u' v(\lceil t_1 \rceil,
			\rho)\downarrow_{u'} c_1$, $\exists c2 \leq u'
			v(\lceil t_2 \rceil, \rho)\downarrow_{u'} c_2$,
			$\exists c3 \leq u' v(\lceil t_3 \rceil,
			\rho)\downarrow_{u'} c_3$ and $c_1 = c_2$, $c_2
			= c_3$.  Hence $c_1 = c_3$.  Since $v(\lceil t_1
			\rceil, \rho)\downarrow_{u'} c_1$ and $v(\lceil
			t_3 \rceil, \rho)\downarrow_{u'} c_3$, $T(u',
			\lceil t_1=t_3 \rceil, \rho)$.  By Lemma
			\ref{lem:T}, $T(u' \oplus r, \lceil t_1=t_3
			\rceil, \rho)$.  Hence $r$ satisfies
			(\ref{eq:soundness}).
			
			$$t_1=t_2 \rightarrow s_i t_1 = s_i t_2$$ where
			$i = 0, 1$.  Assume $T(u', \lceil t_1=t_2
			\rceil, \rho)$.  Then, $\exists c1 \leq u'
			v(\lceil t_1 \rceil, \rho)\downarrow_{u'} c_1$,
			$\exists c2 \leq u' v(\lceil t_2 \rceil, \rho)
			\downarrow_{u'} c_2$ and $c_1 = c_2$.  Hence
			$v(\lceil s_i t_1 \rceil, \rho)\downarrow_{s_i
			u'} s_i c_1$, $v(\lceil s_i t_2 \rceil,
			\rho)\downarrow_{s_i u'} s_i c_2$ and $s_i c_1 =
			s_i c_2$ can be proved by the axiom of $S^2_2$.
			Therefore, $T(s_i u', \lceil s_i t_1 = s_i t_2
			\rceil, \rho)$.  By $s_i u' \leq u' \oplus r$,
			by Lemma \ref{lem:T}, $T(u' \oplus r, \lceil s_i
			t_1 = s_i t_2 \rceil, \rho)$.  Hence $r$
			satisfies (\ref{eq:soundness}).

	     \item[Separation axioms] $$Et \rightarrow t \not= s_1t$$
			Assume $T(u', \lceil Et \rceil, \rho)$
			i.e. $\exists c \leq u' v(\lceil t \rceil, \rho)
			\downarrow_{u'} c$.  Thus $v(\lceil s_1 t
			\rceil, \rho) \downarrow_{s_1 u'} s_1 c$ and by
			clause (2) of Lemma \ref{lem:v1}, $v(\lceil t
			\rceil, \rho) \downarrow_{s_1 u'} c$.  Since
			 $c \not= s_1 c$, $T(s_1 u',
			\lceil t \not= s_1 t \rceil, \rho)$.  Since $s_1
			u' \leq u' \oplus r$ and Lemma \ref{lem:T},
			$T(u' \oplus r, \lceil t \not= s_1 t \rceil,
			\rho)$.

			The cases of $s_0$ ad $Et \rightarrow s_0 t
			\not= s_1 t$ are similar.

	     \item[Inequality axioms] $$Et \rightarrow 0 \leq t$$ Assume
			$T(u', \lceil Et \rceil, \rho)$ i.e. $\exists c
			\leq u' v(\lceil t \rceil, \rho) \downarrow_{u'}
			c$.  $0 \leq c$, and hence $T(u', \lceil 0 \leq t
			\rceil, \rho)$.  By Lemma \ref{lem:T}, $T(u'
			\oplus r, \lceil 0 \leq t \rceil, \rho)$.

			$$t_1 \leq t_2 \rightarrow s_i t_1 \leq s_i
			t_2$$ where $i = 0, 1$.  Assume $T(u', \lceil
			t_1 \leq t_2 \rceil, \rho)$.  Then, $\exists c_1
			\leq u' v(\lceil t_1 \rceil, \rho)
			\downarrow_{u'} c_1$, $\exists c_2 \leq u'
			v(\lceil t_2 \rceil, \rho) \downarrow_{u'} c_2$
			and $c_1 \leq c_2$.  Hence, $v(\lceil s_i t_1
			\rceil, \rho) \downarrow_{s_i u'} s_i c_1$ and
			$v(\lceil s_i t_2 \rceil, \rho) \downarrow_{s_i
			u'} s_i c_2$.  Since $s_i c_1 \leq s_i c_2$,
			$T(s_i u', \lceil s_i t_1 \leq s_i t_2 \rceil,
			\rho)$.  Since $s_i u' \leq u' \oplus r$, $T(u'
			\oplus r, \lceil s_i t_1 \leq s_i t_2 \rceil,
			\rho)$.
			
			$$t_1 \leq t_2 \rightarrow s_0 t_1 \leq s_1
			t_2$$ where $i = 0, 1$.  Assume $T(u', \lceil
			t_1 \leq t_2 \rceil, \rho)$.  Then, $\exists c_1
			\leq u' v(\lceil t_1 \rceil, \rho)
			\downarrow_{u'} c_1$, $\exists c_2 \leq u'
			v(\lceil t_2 \rceil, \rho) \downarrow_{u'} c_2$
			and $c_1 \leq c_2$.  Hence, $v(\lceil s_0 t_1
			\rceil, \rho) \downarrow_{s_1 u'} s_0 c_1$ and
			$v(\lceil s_1 t_2 \rceil, \rho) \downarrow_{s_1
			u'} s_1 c_2$.  Since $s_0 c_1 \leq s_1 c_2$,
			$T(s_1 u', \lceil s_0 t_1 \leq s_1 t_2 \rceil,
			\rho)$.  Since $s_1 u' \leq u' \oplus r$, $T(u'
			\oplus r, \lceil s_0 t_1 \leq s_1 t_2 \rceil,
			\rho)$.
	     \item[Defining axioms $Cond$] $$Et_1, Et_2 \rightarrow Cond
			(0, t_1, t_2) = t_1$$ Assume that $T(u', \lceil
			Et_1 \rceil, \rho)$ and $T(u', \lceil Et_2
			\rceil, \rho)$.  Therefore we have a
			$\rho$-evaluation tree $w_1$ of term $t_1$ and
			$w_2$ of $t_2$.  Using $w_1$ and $w_2$, we can
			construct $\rho$-evaluation tree of $Cond(0,
			t_1, t_2)$.  Hence, $\exists c \leq u' v(\lceil
			Cond(0, t_1, t_2) \rceil, \rho) \downarrow_{u'}
			c$.  By clause (1) of Lemma \ref{lem:v2},
			$v(\lceil t_1 \rceil, \rho) \downarrow_{u'} d$
			and $d = c$.  Hence, $T(u', \lceil Cond (0, t_1,
			t_2) = t_1 \rceil, \rho)$ holds.  Since $u' \leq
			u' \oplus r$ and Lemma \ref{lem:T}, we have
			$T(u' \oplus r, \lceil Cond (0, t_1, t_2) = t_1
			\rceil, \rho)$.

			$$ECond(t_1, t_2, t_3) \rightarrow Cond (s_0
			t_1, t_2, t_3) = Cond (t_1, t_2, t_3)$$ Assume
			that $T(u', \lceil ECond(t_1, t_2, t_3) \rceil,
			\rho)$.  Then, we have a $\rho$-valuation tree
			of $t_1, t_2, t_3$ bounded by
			$u'$ respectively.  Therefore, we can construct
			a $\rho$-valuation tree of $Cond (s_0 t_1,
			t_2, t_3)$ bounded by $s_0 u'$.  Hence we have
			$\exists c_1 \leq s_0 u' v(\lceil Cond (s_0 t_1,
			t_2, t_3) \rceil, \rho)\downarrow_{s_0u'} c_1$.
			Moreover $\exists c_2 \leq s_0 u' v(\lceil Cond
			(t_1, t_2, t_3) \rceil, \rho)\downarrow_{s_0u'}
			c_2$ by clause (2) of Lemma \ref{lem:v1}. By
			clause (1) of Lemma \ref{lem:v2}, $c_1 = c_2$.
			Hence $T(s_0 u' \lceil Cond (s_0 t_1, t_2, t_3)
			= Cond (t_1, t_2, t_3) \rceil, \rho)$.  By $s_0
			u' \leq u' \oplus r$ and Lemma \ref{lem:T},
			we have $T(u' \oplus r, \lceil Cond (s_0 t_1, t_2, t_3)
			= Cond (t_1, t_2, t_3) \rceil, \rho)$.

			$$Et_1, Et_2, Et_3 \rightarrow Cond (s_1 t_1,
			t_2, t_3) = t_3$$ Assume $T(u', \lceil Et_1
			\rceil, \rho)$, $T(u', \lceil Et_2 \rceil,
			\rho)$, $T(u', \lceil Et_3 \rceil, \rho)$.
			Then, we have a $\rho$-valuation $w_1, w_2, w_3$
			tree of $t_1, t_2, t_3$ bounded by $u'$
			respectively.  Hence, 
			$\rho$-valuation tree $w$ of $Cond(s_1 t_1, t_2,
			t_3)$ bounded by $s_1 u'$ can be constructed
			from $w_1, w_2, w_3$.  Let $c$ be the value
			of $w$.  By clause (1) of Lemma \ref{lem:v1},
			$w_3$ is still a $\rho$-valuation tree bounded
			by $s_1 u'$.  Hence $T(s_1 u', \lceil Cond (s_1
			t_1, t_2, t_3) = t_3 \rceil, \rho)$.  Since $s_1
			u' \leq u' \oplus r$ and Lemma \ref{lem:T},
			$T(u' \oplus r, \lceil Cond (s_1 t_1, t_2, t_3)
			= t_3 \rceil, \rho)$.

	     \item[Defining axioms $S$] 
			  $$\rightarrow S0 = s_10$$  Since $T(1, \lceil
			S0 = s_10 \rceil, \rho)$, we have done.

			  $$Es_1t \rightarrow Ss_0t = s_1t$$

			By $T(u', \lceil Es_1t \rceil, \rho)$, we have
			an $\rho$-valuation tree $w$ of $s_1 t$ bounded
			by $u'$.  Using $w$, we can construct
			$\rho$-valuation tree of $Ss_0t$ and $s_1t$
			bounded by $u'$.  By clause (1) of Lemma
			\ref{lem:v2}, values of both trees are equal.
			Hence $T(u', \lceil Ss_0t = s_1t \rceil, \rho)$.
			Since $u' \leq u' \oplus r$, by Lemma \ref{lem:T},
			we have done.

			$$ESt \rightarrow Ss_1t = s_0(St)$$

			By $T(u', \lceil ESt \rceil, \rho)$, we have an
			$\rho$-valuation tree $w$ of $S t$ bounded by
			$u'$.  Using $w$, we can construct
			$\rho$-valuation tree $w_1$ of $Ss_1t$ and $w_2$
			of $s_0(St)$ bounded by $s_0 u'$.  By clause (1)
			of Lemma \ref{lem:v2}, values of both trees are
			equal.  Hence $T(s_0 u', \lceil Ss_1t = s_0(St)
			\rceil,\rho)$.  Since $s_0 u' \leq u' \oplus r$,
			by Lemma \ref{lem:T}, we have done.

	     \item[Defining axioms $|\ |$] 

			$$\rightarrow |0|=0$$

			Since $T(0, \lceil |0|=0 \rceil, \rho)$, by
			Lemma \ref{lem:T} we have done.

			$$ES|t| \rightarrow |s_0t| = Cond (t, 0,
			S|t|)$$

			By $T(u', \lceil ES|t| \rceil, \rho)$, we have
			an $\rho$-valuation tree $w$ of $t$ bounded by
			$u'$.  From $w$, we can construct
			$\rho$-valuation tree $w_1$ of $|s_0 t|$ and
			$w_2$ of $Cond(t, 0, S|t|)$ bounded by $s_0 u'$.
			By clause (1) of Lemma \ref{lem:v2}, the values of
			$w_1$ and $w_2$ are equal.  Hence,
			$T(s_0 u', \lceil |s_0t| = Cond (t, 0, S|t|)
			\rceil, \rho)$.  Since $s_0 u' \leq u \oplus r$,
			by Lemma \ref{lem:T} $T(u' \oplus r, \lceil
			|s_0t| = Cond (t, 0, S|t|) \rceil, \rho)$.

			$$ES|t| \rightarrow |s_1t| = S|t|$$

			Analogous to the proof above.

	     \item[Defining axioms $\lfloor \frac{}{2} \rfloor$] 
			$$\rightarrow \lfloor \frac{0}{2} \rfloor = 0$$

			Since $T(2, \lceil\lfloor \frac{0}{2} \rfloor =
			0 \rceil, \rho)$, we have done.

			$$Et \rightarrow \lfloor \frac{1}{2} s_0t
			  \rfloor = t$$

			Assume $T(u', \lceil Et \rceil, \rho)$.  Then,
			we have an $\rho$-valuation tree $w$ of $t$
			bounded by $u'$.  From $w$, we can construct
			$\rho$-valuation tree $w_1$ of $\lfloor \frac{1}{2}
			s_0t \rfloor$ bounded by $s_0 u'$.  By clause (1) of
			Lemma \ref{lem:v1}, $w$ is a $\rho$-valuation
			tree bounded by $s_0 u'$.  By clause (1) of
			Lemma \ref{lem:v2}, the values of $w_1$ and $w$
			are equal.  Hence $T(s_0 u', \lceil \lfloor
			\frac{1}{2} s_0t \rfloor = t \rceil, \rho)$.
			Since $s_0 u' \leq u \oplus r$, by Lemma
			\ref{lem:T} $T(u' \oplus r, \lceil \lfloor
			\frac{1}{2} s_0t \rfloor = t \rceil, \rho)$.

			$$Et \rightarrow \lfloor \frac{1}{2} s_1t
			\rfloor = t$$

			Analogous to the proof above.

	     \item[Defining axioms $\boxplus$] 

			$$Et \rightarrow t \boxplus 0 = t$$

			Assume $T(u', \lceil Et \rceil, \rho)$.  Then,
			we have an $\rho$-valuation tree $w$ of $t$
			bounded by $u'$.  Hence we have an
			$\rho$-valuation tree $w_1$ of $t \boxplus 0$
			bounded by $u'$.  By clause (1) of Lemma
			\ref{lem:v2}, both values of $w$ and $w_2$ are
			equal.  Hence $T(u', \lceil t \boxplus 0 = t
			\rceil, \rho)$.  By Lemma \ref{lem:T}, $T(u'
			\oplus r, \lceil t \boxplus 0 = t
			\rceil, \rho)$.

			$$Es_0(t_1 \boxplus t_2) \rightarrow t_1 \boxplus s_0
			t_2 = Cond (t_2, t_1, s_0 (t_1 \boxplus t_2))$$

			Assume $T(u', \lceil Es_0(t_1 \boxplus t_2)
			\rceil, \rho)$.  Then, we have an
			$\rho$-valuation tree $w$ of $Es_0(t_1 \boxplus
			t_2)$ bounded by $u'$.  By manipulating $w$, we
			can construct $\rho$-valuation tree $w_1$ of
			$t_1 \boxplus s_0 t_2$ and $w_2$ of $Cond (t_2,
			t_1, s_0 (t_1 \boxplus t_2))$ bounded by $u'$.
			By clause (1) of Lemma \ref{lem:v2}, the values
			of $w_1$ and $w_2$ are equal.  Hence $T(u',
			\lceil t_1 \boxplus s_0 t_2 = Cond (t_2, t_1,
			s_0 (t_1 \boxplus t_2)) \rceil, \rho)$.  By $u'
			\leq u' \oplus r$ and Lemma \ref{lem:T}, $T(u'
			\oplus r, \lceil t_1 \boxplus s_0 t_2 = Cond
			(t_2, t_1, s_0 (t_1 \boxplus t_2)) \rceil,
			\rho)$.

			$$Es_0(t_1 \boxplus t_2) \rightarrow t_1 \boxplus s_1
			t_2 = s_0 (t_1 \boxplus t_2)$$
			
			Analogous to the proof above.

	     \item[Defining axioms $\#$] 

			$$Et \rightarrow t \# 0 = 1$$

			Assume $T(u', \lceil Et \rceil, \rho)$.  Then,
			we have an $\rho$-valuation tree $w$ of $t$
			bounded by $u'$.  From $w$, we can construct
			$\rho$-valuation tree $w_1$ of $t \# 0$ bounded
			by $u'$.  By axioms (of $S^2_2$), the value of
			$w_1$ is $1$.  The valuation tree $w_2$ of $1$
			is bounded by $1$.  Hence $T(max\{u', 1\}, \lceil
			t \# 0 = 1 \rceil. \rho)$.  Since $max\{u', 1\}
			\leq u' \oplus \rho$, by Lemma \ref{lem:T} $T(u'
			\oplus r, \lceil t \# 0 = 1 \rceil. \rho)$.

			$$E (t_1 \# t_2) \boxplus t_1 \rightarrow t_1
			\# s_0 t_2 = Cond (t_2, 1, (t_1 \# t_2) \boxplus
			t_1)$$

			Assume $T(u', \lceil E (t_1 \# t_2) \boxplus t_1
			\rceil, \rho)$.  Then, we have an
			$\rho$-valuation tree $w$ of $(t_1 \# t_2)
			\boxplus t_1$ bounded by $u'$.  Manipulating
			$w$, we have $\rho$-valuation tree $w_1$ of $t_1
			\# s_0 t_2$ bounded by $max\{s_0 u', 1\}$ and a
			$\rho$-valuation tree $w_2$ of $Cond (t_2, 1,
			(t_1 \# t_2) \boxplus t_1)$ bounded by $max\{s_0
			u', 1\}$.  By clause (1) of Lemma \ref{lem:v2}
			and axioms of $S^2_2$, the values of $w_1$ and
			$w_2$ are equal.  Hence $T(max\{s_0 u', 1\},
			\lceil t_1 \# s_0 t_2 = Cond (t_2, 1, (t_1 \#
			t_2) \boxplus t_1) \rceil, \rho)$.  Since
			$max\{s_0 u', 1\} \leq u' \oplus r$, by Lemma
			\ref{lem:T} we have $T(u' \oplus r, \lceil t_1
			\# s_0 t_2 = Cond (t_2, 1, (t_1 \# t_2) \boxplus
			t_1) \rceil, \rho)$.

			$$E (t_1 \# t_2) \boxplus t_1 \rightarrow t_1 \#
			s_1t_2 = (t_1 \# t_2) \boxplus t_1 $$

			Analogous to the proof above.

	     \item[Defining axioms $parity$] 
			$$\rightarrow parity(0) = 0$$
	
			Since $T(0, \lceil parity(0) = 0 \rceil, \rho)$,
			we have done.

			$$Et \rightarrow parity(s_0 t) = 0$$.

			Assume $T(u', \lceil Et \rceil, \rho)$.  Then,
			we have an $\rho$-valuation tree $w$ of $t$
			bounded by $u'$.  From $w$, we can construct
			$\rho$-valuation tree $w_1$ of $parity(s_0 t)$
			bounded by $s_0 u'$.  By clause (1) of Lemma
			\ref{lem:v2} we can reason that the value of
			$w_1$ equals 0.  Hence we have $T(s_0 u', \lceil
			parity(s_0 t) = 0 \rceil, \rho)$.  Since $s_0 u'
			\leq u' \oplus r$ and by Lemma \ref{lem:T},
			$T(u' \oplus r, \lceil
			parity(s_0 t) = 0 \rceil, \rho)$.

			$$Et \rightarrow parity(s_1 t) = 1$$

			Assume $T(u', \lceil Et \rceil, \rho)$.  Then,
			we have an $\rho$-valuation tree $w$ of $t$
			bounded by $u'$.  From $w$, we can construct
			$\rho$-valuation tree $w_1$ of $parity(s_1 t)$
			bounded by $max\{s_1 u', 1\}$.  By clause (1) of
			Lemma \ref{lem:v2}, the value of $w_1$ equals 1.
			Hence we have $T(max\{s_1 u', 1\}, \lceil
			parity(s_0 t) = 0 \rceil, \rho)$.  Since
			$max\{s_1 u', 1\} \leq u' \oplus r$, $T(u'
			\oplus r, \lceil parity(s_1 t) = 1 \rceil,
			\rho)$ by Lemma \ref{lem:T}.

	     \item[Defining axioms $+$] 
			
			$$Et \rightarrow t + 0 = t$$

			Assume $T(u', \lceil Et \rceil, \rho)$.  Then,
			we have an $\rho$-valuation tree $w$ of $t$
			bounded by $u'$.  From $w$, we can construct
			$\rho$-valuation tree $w_1$ of $t + 0$ bounded
			by $u'$.  By axioms (of $S^2_2$), the value of
			$w_1$ equals to $w$. Hence $T(u', \lceil
			t + 0 = t \rceil. \rho)$.  Since $u'
			\leq u' \oplus \rho$, by Lemma \ref{lem:T} $T(u'
			\oplus r, \lceil t + 0 = t \rceil, \rho)$.
			
			$$E(\lfloor \frac{1}{2} t_1 \rfloor + t_2)
			  \rightarrow t_1 + s_0 t_2 = Cond(parity(t_1),
			  s_0(\lfloor \frac{1}{2} t_1 \rfloor + t_2) ,
			  s_1(\lfloor \frac{1}{2} t_1 \rfloor + t_2))$$

			Assume $T(u',\lceil E(\lfloor \frac{1}{2} t_1
			\rfloor + t_2) \rceil, \rho)$.  Then we have
			$\rho$-valuation tree $w$ of $\lfloor
			\frac{1}{2} t_1 \rfloor + t_2$ bounded by $u'$.
			Then, we can construct $\rho$-valuation trees
			$w_1$ of $t_1 + s_0 t_2$ and $w_2$ of
			$Cond(parity(t_1), s_0(\lfloor \frac{1}{2} t_1
			\rfloor + t_2) , s_1(\lfloor \frac{1}{2} t_1
			\rfloor + t_2))$ bounded by $s_1 u'$.  By clause
			(1) of Lemma \ref{lem:v2}, the values of $w_1$
			and $w_2$ are equal.  Hence $T(s_1 u',
			\lceil t_1 + s_0 t_2 = Cond(parity(t_1),
			  s_0(\lfloor \frac{1}{2} t_1 \rfloor + t_2) ,
			  s_1(\lfloor \frac{1}{2} t_1 \rfloor + t_2))
			\rceil, \rho)$.  Since $s_1 u' \leq u' \oplus
			r$, by Lemma \ref{lem:T}, $T(u' \oplus r,
			\lceil t_1 + s_0 t_2 = Cond(parity(t_1),
			  s_0(\lfloor \frac{1}{2} t_1 \rfloor + t_2) ,
			  s_1(\lfloor \frac{1}{2} t_1 \rfloor + t_2))
			\rceil, \rho)$.

			$$E(\lfloor \frac{1}{2} t_1 \rfloor + t_2)
			\rightarrow t_1 + s_1 t_2 = Cond(parity(t_1),
			s_1(\lfloor \frac{1}{2} t_1 \rfloor + t_2) ,
			s_0(S(\lfloor \frac{1}{2} t_1 \rfloor + t_2)))$$

			Assume $T(u',\lceil E(\lfloor \frac{1}{2} t_1
			\rfloor + t_2) \rceil, \rho)$.  Then we have
			$\rho$-valuation tree $w$ of $\lfloor
			\frac{1}{2} t_1 \rfloor + t_2$ bounded by $u'$.
			Then, we can construct $\rho$-valuation trees
			$w_1$ of $t_1 + s_1 t_2$ and $w_2$ of
			$Cond(parity(t_1), s_1(\lfloor \frac{1}{2} t_1
			\rfloor + t_2) , s_0(S(\lfloor \frac{1}{2} t_1
			\rfloor + t_2)))$ bounded by $s_0(S u')$.  Hence
			we have $T(s_0(S u'), \lceil t_1 + s_1 t_2 =
			Cond(parity(t_1), s_1(\lfloor \frac{1}{2} t_1
			\rfloor + t_2) , s_0(S(\lfloor \frac{1}{2} t_1
			\rfloor + t_2))) \rceil, \rho)$.  By $s_0(S u')
			\leq u' \oplus r$ and Lemma \ref{lem:T}, $T(u'
			\oplus r, \lceil t_1 + s_1 t_2 =
			Cond(parity(t_1), s_1(\lfloor \frac{1}{2} t_1
			\rfloor + t_2) , s_0(S(\lfloor \frac{1}{2} t_1
			\rfloor + t_2))) \rceil, \rho)$.

	     \item[Defining axioms $\cdot$] 
			$$Et \rightarrow t \cdot 0 = 0$$

			Assume $T(u', \lceil Et \rceil, \rho)$.  Then,
			we have an $\rho$-valuation tree $w$ of $t$
			bounded by $u'$.  From $w$, we can construct
			$\rho$-valuation tree $w_1$ of $t \cdot 0$ bounded
			by $u'$.  By axioms (of $S^2_2$), the value of
			$w_1$ equals to $0$. Hence $T(u', \lceil
			t \cdot 0 = 0 \rceil. \rho)$.  Since $u'
			\leq u' \oplus \rho$, by Lemma \ref{lem:T} $T(u'
			\oplus r, \lceil t \cdot 0 = t \rceil, \rho)$.
			
			$$E t_1 \cdot t_2 \rightarrow t_1 \cdot (s_0 t_2) =
			s_0(t_1 \cdot t_2)$$

			Assume $T(u', \lceil E t_1 \cdot t_2 \rceil,
			\rho)$.  Then, we have $\rho$-valuation tree $w$
			of $t_1 \cdot t_2$ bounded by $u'$.  Hence we
			have $\rho$-valuation trees $w_1$ of $t_1 \cdot
			(s_0 t_2)$ and $w_2$ of $s_0(t_1 \cdot t_2)$
			bounded by $s_0 u'$.  By clause (1) of Lemma
			\ref{lem:v2} and axioms of $S^2_2$, the values of
			$w_1$ and $w_2$ are equal.  Hence $T(s_0 u',
			\lceil t_1 \cdot (s_0 t_2) = s_0(t_1 \cdot t_2)
			\rceil, \rho)$.  Since $s_0 u' \leq u' \oplus
			r$, by Lemma \ref{lem:T} $T(u' \oplus r, \lceil
			t_1 \cdot (s_0 t_2) = s_0(t_1 \cdot t_2) \rceil,
			\rho)$.

			$$Es_0(t_1 \cdot t_2) + t_1 \rightarrow t_1 \cdot (s_1
			t_2) = s_0(t_1 \cdot t_2) + t_1$$

			Assume that $T(u', \lceil Es_0(t_1 \cdot t_2) +
			t_1 \rceil, \rho)$.  Then we have
			$\rho$-evaluation $w$ of term $s_0(t_1 \cdot
			t_2) + t_1$ bounded by $u'$.  By rearranging
			tree $w$ and using axioms of $S^2_2$, we have
			$\rho$-valuation trees $w'$ of term $t_1 \cdot
			(s_1 t_2)$ bounded by $u'$.  By axioms of
			$S^2_2$, values of $w$ and $w'$ are equal.
			Hence, $T(u', \lceil t_1 \cdot (s_1 t_2) =
			s_0(t_1 \cdot t_2) + t_1 \rceil, \rho)$.  By
			Lemma \ref{lem:T}, $T(u' \oplus r, \lceil t_1
			\cdot (s_1 t_2) = s_0(t_1 \cdot t_2) + t_1
			\rceil, \rho)$.
	    \end{description}
\item[Structural rules]     
	    \begin{description}
	     \item[Weakening]
			
			$$\infer{A, \Gamma \rightarrow
			\Delta}{\Gamma \rightarrow \Delta}$$

			By induction hypothesis, (\ref{eq:soundness})
			holds for the assumption.  Hence
			(\ref{eq:soundness}) trivially holds for the
			conclusion. 
	
			$$\infer{\Gamma \rightarrow
			 \Delta, A}{\Gamma \rightarrow \Delta}$$

			By induction hypothesis, (\ref{eq:soundness})
			holds for the assumption.  Hence
			(\ref{eq:soundness}) trivially holds for the
			conclusion. 

	     \item[Contraction] 
			$$\infer{A, \Gamma \rightarrow
			\Delta}{A, A, \Gamma \rightarrow \Delta}$$
			
			By induction hypothesis, (\ref{eq:soundness})
			holds for the assumption.  Hence
			(\ref{eq:soundness}) trivially holds for the
			conclusion.

			$$\infer{\Gamma \rightarrow
			\Delta, A}{\Gamma \rightarrow \Delta, A, A}$$

			By induction hypothesis, (\ref{eq:soundness})
			holds for the assumption.  Hence
			(\ref{eq:soundness}) trivially holds for the
			conclusion. 
			
	      \item[Exchange] 
			$$\infer{\Gamma, B, A, \Pi \rightarrow
			\Delta}{\Gamma, A, B, \Pi \rightarrow \Delta}$$

			By induction hypothesis, (\ref{eq:soundness})
			holds for the assumption.  Hence
			(\ref{eq:soundness}) trivially holds for the
			conclusion. 

			$$\infer{\Gamma \rightarrow \Delta, B, A,
			\Pi}{\Gamma \rightarrow \Delta, A, B, \Pi}$$

			By induction hypothesis, (\ref{eq:soundness})
			holds for the assumption.  Hence
			(\ref{eq:soundness}) trivially holds for the
			conclusion. 
	    \end{description}
  \item[Logical rules]
	     \begin{description}
	      \item[$\neg$-rules] $$\infer{\neg p(t_1, \dots, t_n),
			 \Gamma \rightarrow \Delta}{\infer*[r_1]{\Gamma
			 \rightarrow \Delta, p(t_1, \dots, t_n)}{}}$$ By
			 induction hypothesis, (\ref{eq:soundness})
			 holds for the assumption.  To prove
			 (\ref{eq:soundness}) for the conclusion, we
			 first assume that $\forall A \in \Gamma \ T(u',
			 \lceil A \rceil, \rho)$ and $T(u',\lceil \neg
			 p(t_1, \dots, t_n) \rceil,\rho)$.  By induction
			 hypothesis, Either $\exists B \in \Delta \ T(u'
			 \oplus r_1, \lceil B \rceil,\rho)$ or $T(u'
			 \oplus r_1, \lceil p(t_1, \dots, t_n)
			 \rceil,\rho)$.  But from hypothesis and Lemma
			 \ref{lem:T}, $T(u' \oplus r_1, \lceil p(t_1,
			 \dots, t_n) \rceil,\rho)$.  Hence, if we have
			 $T(u',\lceil \neg p(t_1, \dots, t_n)
			 \rceil,\rho)$ then contradiction.  Hence
			 $\exists B \in \Delta \ T(u' \oplus r_1, \lceil
			 B \rceil,\rho)$.  Since $u' \oplus r_1 \leq u'
			 \oplus r$, by Lemma \ref{lem:T}, we have done.

			 $$\infer{Et_1, \dots, Et_n, \Gamma \rightarrow
			 \Delta, \neg p(t_1, \dots,
			 t_n)}{\infer*[r_1]{p(t_1, \dots, t_n), \Gamma
			 \rightarrow \Delta}{}}$$

			 Assume $\forall A \in \Gamma \ T(u', \lceil A
			 \rceil, \rho)$ and $T(u', \lceil Et_i
			 \rceil,\rho)$ for all $i = 1, \dots, n$.  If
			 $\exists B \in \Delta T(u' \oplus r, \lceil A
			 \rceil, \rho)$, we have done.  So assume
			 otherwise.  Then $\forall B \in \Delta$, $T(u'
			 \oplus r_1, \lceil A \rceil, \rho)$ does not
			 hold.  Hence, by induction hypothesis, $T(u',
			 \lceil p(t_1, \dots, t_n) \rceil, \rho)$ does
			 not hold.  Therefore, by Lemma \ref{lem:EM} and
			 hypothesis, $T(u',
			 \lceil \neg p(t_1, \dots, t_n) \rceil, \rho)$
			 does holds.  Hence we have done.
			 
	      \item[$\wedge$-rules] $$\infer{A \wedge B. \Gamma
			 \rightarrow \Delta}{\infer*[r_1]{A, \Gamma
			 \rightarrow \Delta}{}}$$

			 By Lemma \ref{lem:T_0}, $T(u', \lceil A \wedge
			 B \rceil, \rho)$ implies $T(u', \lceil A
			 \rceil, \rho)$.  Hence, by induction
			 hypothesis, $\exists C \in \Delta$ such that
			 $T(u' \oplus r_1, \lceil C \rceil, \rho)$.
			 Since $u' \oplus r_1 \leq u' \oplus r$, by
			 Lemma \ref{lem:T}, we have done.

			 $$\infer{B \wedge A. \Gamma
			 \rightarrow \Delta}{A, \Gamma \rightarrow
			 \Delta}$$ 

			 This case is proved similarly as above.

			 $$\infer{\Gamma \rightarrow \Delta, A \wedge
			 B}{\infer*[r_1]{\Gamma \rightarrow \Delta, A
			 }{} \quad \infer*[r_2]{\Gamma \rightarrow \Delta, B}{}}$$

			 Assume $\forall C \in \Gamma \ T(u', \lceil C
			 \rceil, \rho)$. By induction hypothesis, either
			 $\exists D \in \Delta, T(u' \oplus r_1, \lceil
			 D \rceil, \rho)$ or $T(u' \oplus r_1, \lceil A
			 \rceil, \rho)$.  For the former case, since $u'
			 \oplus r_1 \leq u' \oplus r$, by Lemma
			 \ref{lem:T}, we have done.  Otherwise, $T(u'
			 \oplus r_1, \lceil A \rceil, \rho)$.  By
			 induction hypothesis, either $\exists D \in
			 \Delta, T(u' \oplus r_2, \lceil D \rceil,
			 \rho)$ or $T(u' \oplus r_2, \lceil B \rceil,
			 \rho)$.  For the former case, again since $u'
			 \oplus r_2 \leq u' \oplus r$, we have done.
			 Otherwise, $T(u' \oplus r_2, \lceil B \rceil,
			 \rho)$.  Since $u' \oplus r_1, u' \oplus r_2
			 \leq u' \oplus r$, $T(u' \oplus r, \lceil A
			 \rceil, \rho)$ and $T(u' \oplus r, \lceil B
			 \rceil, \rho)$.  Hence, by Lemma \ref{lem:T_0}
			 and the definition of $T$,  $T(u' \oplus r,
			 \lceil A \wedge B \rceil, \rho)$.  

	      \item[$\vee$-rules] $$\infer{A \vee B, \Gamma \rightarrow
			 \Delta}{\infer*[r_1]{A, \Gamma \rightarrow
			 \Delta}{} \quad B, \infer*[r_2]{\Gamma
			 \rightarrow \Delta}{}}$$

			 It suffices to show that if $T(u', \lceil A
			 \vee B \rceil, \rho)$ and $\forall C \in \Gamma,
			 T(u', \lceil C \rceil, \rho)$ then $\exists D
			 \in \Delta, T(u', \lceil D \rceil, \rho)$.
			 Assume that $T(u', \lceil A \vee B \rceil,
			 \rho)$ and $\forall C \in \Gamma, T(u', \lceil C
			 \rceil, \rho)$.  By definition of $T$ and Lemma
			 \ref{lem:T_0}, $T(u', \lceil A \vee B \rceil,
			 \rho)$ is equivalent to $T(u', \lceil A \rceil,
			 \rho)$ or $T(u', \lceil B \rceil, \rho)$.
			 Hence, by induction hypothesis, either $\exists
			 D \in \Delta T(u' \oplus r_1, \lceil D \rceil)$
			 or $\exists D \in \Delta T(u' \oplus r_2, \lceil
			 D \rceil)$.  Since $u' \oplus r_1, u' \oplus r_2
			 \leq u' \oplus r$, we have $\exists D \in \Delta
			 T(u' \oplus r, \lceil D \rceil)$
			 
			 $$\infer{\Gamma \rightarrow \Delta, A \vee
			 B}{\infer*[r_1]{\Gamma \rightarrow \Delta, A}{}}$$

			 Assume $\forall C \in \Gamma, T(u', \lceil C
			 \rceil, \rho)$.  By induction hypothesis,
			 either $\exists D \in \Delta, T(u' \oplus r_1,
			 \lceil D \rceil, \rho)$ or $T(u' \oplus r_1,
			 \lceil A \rceil, \rho)$.  If $\exists D \in
			 \Delta, T(u' \oplus r_1, \lceil D \rceil,
			 \rho)$, then since $u' \oplus r_1 \leq u' \oplus
			 r$, we have done.  Otherwise, $T(u' \oplus r_1,
			 \lceil A \rceil, \rho)$.  By definition of $T$
			 and Lemma \ref{lem:T_0}, we have $T(u' \oplus r_1,
			 \lceil A \vee B \rceil, \rho)$.  Since $u'
			 \oplus r_1 \leq u' \oplus r$, we have done.

			 $$\infer{\Gamma \rightarrow \Delta, B \vee
			 A}{\Gamma \rightarrow \Delta, A}$$

			 The proof is similar as above.

	      \item[$\forall$-rules] 

			 $$\infer{t \leq s, \forall x \leq s.  A(x),
			 \Gamma \rightarrow \Delta}{infer*[r_1]{A(t),
			 \Gamma \rightarrow \Delta}{}}$$

			 Assume that $C$ satisfies $T(u', \lceil C
			 \rceil, \rho)$ if $C$ is a formula in $t \leq
			 s, \forall x \leq s.  A(x), \Gamma$.  Then,
			 there are $c, d$ such that $v(\lceil t \rceil)
			 \downarrow_u' c$, $v(\lceil s \rceil)
			 \downarrow_u' d$ and $c \leq d$.  Since the
			 proof $w$ is a 1-normal proof, in $\forall x
			 \leq s.  A(x)$, $s$ has a form $|s'|$.  By
			 assumption, $T(u', \lceil \forall x \leq s.
			 A(x) \rceil, \rho)$.  By Definition
			 \ref{defn:T}, $\exists d' \leq u', v(u', \lceil
			 s' \rceil, \rho)$ and $\forall x \leq |d'|,
			 T_0(u', \lceil A(x) \rceil, \rho*x)$.  By Lemma
			 \ref{lem:v2}, $d = |d'|$.  Hence, $c \leq
			 |d'|$.  Therefore, $T_0(u', \lceil A(x) \rceil,
			 \rho*c)$.  By Lemma \ref{lem:T_0}, $T_0(u',
			 \lceil A(t) \rceil, \rho)$.  By Definition
			 \ref{defn:T}, $T(u', \lceil A(t) \rceil, \rho)$.
			 Combining the fact that $\forall C \in \Gamma,
			 T(u', \lceil C \rceil, \rho)$, by induction
			 hypothesis of $r_1$, $\exists D \in \Delta, T(u'
			 \oplus r_1, \lceil D \rceil, \rho)$.  Since $u'
			 \oplus r_1 \leq u' \oplus r$ and by Lemma
			 \ref{lem:T}, $T(u' \oplus r, \lceil D \rceil,
			 \rho)$

			 $$\infer{Nt, \Gamma \rightarrow \Delta, \forall
			 x \leq t. A(x)}{\infer*[r_1]{x \leq t, \Gamma
			 \rightarrow \Delta, A(x)}{}}$$ where $x$ does
			 not appear in $\Gamma, \Delta$ and $t$.

			 Assume $T(u', \lceil Nt \rceil, \rho)$ and
			 $\forall C \in \Gamma, T(u', \lceil C \rceil,
			 \rho)$.  Then, by the first assumption,
			 $\exists c \leq u'. v(u',\lceil t \rceil,
			 \rho)\downarrow_u' c$.  let $d$ be any natural
			 number satisfying $d \leq c$.  Then, $T(u',
			 \lceil x \leq t \rceil, \rho*d)$ and since $x$
			 does not occur in $\Gamma$, $\forall C \in
			 \Gamma, T(u', \lceil C \rceil, \rho*d)$.  By
			 induction hypothesis on $r_1$, either $\exists
			 D \in \Delta T(u' \oplus r_1, \lceil D \rceil,
			 \rho*d)$ or $T(u', \lceil A(x) \rceil, \rho*d)$.

	      \item[$\exists$-rules] $$\infer{\exists x \leq t. A(x),
			 \Gamma \rightarrow \Delta}{\infer*[r_1]{x \leq
			 t, A(x), \Gamma \rightarrow \Delta}{}}$$ where
			 $x$ does not appear in $\Gamma, \Delta$.

			 Assume $T(u', \lceil \exists x \leq t. A(x)
			 \rceil, \rho)$ and $\forall C \in \Gamma, T(u',
			 \lceil C \rceil, \rho)$.  By definition of $T$,
			 $\exists c \leq u'$ such that $v(\lceil t
			 \rceil, \rho) \downarrow_u' c$ and $\exists d
			 \leq c, T(u', \lceil A(x) \rceil , \rho*d)$.
			 Since $x$ does not appears in $\Gamma$,
			 $\forall C \in \Gamma, T(u', \lceil C \rceil,
			 \rho*d)$.  By induction hypothesis on $r_1$,
			 $\exists D \in \Delta, T(u' \oplus r_1, \lceil D
			 \rceil, \rho*d)$.  Since $D$ does not have $x$
			 as a free variable, $T(u' \oplus r_1, \lceil D
			 \rceil, \rho)$.  Since $u' \oplus r_1 \leq u'
			 \oplus r$ and by Lemma \ref{lem:T}, $T(u' \oplus
			 r, \lceil D \rceil, \rho)$.

			 $$\infer{t \leq s, \Gamma \rightarrow \Delta,
			 \exists x \leq s. A(x)}{\infer*[r_1]{\Gamma
			 \rightarrow \Delta, A(t)}{}}$$

			 Assume that $T(u', \lceil t \leq s \rceil,
			 \rho)$ and $\forall C \in \Gamma, T(u', \lceil C
			 \rceil, \rho)$.  By induction hypothesis on
			 $r_1$, either $\exists D \in \Delta, T(u' \oplus
			 r_1, \lceil D \rceil, \rho)$ or $T(u' \oplus
			 r_1, \lceil A(t) \rceil, \rho)$.  If $\exists D
			 \in \Delta, T(u' \oplus r_1, \lceil D \rceil,
			 \rho)$, we have done by Lemma \ref{lem:T_0}.
			 Hence, assume $T(u' \oplus r_1, \lceil A(t)
			 \rceil, \rho)$.  By Lemma \ref{lem:v2},
			 $\exists c. v(\lceil t \rceil,
			 \rho)\downarrow_{u' \oplus r_1} c$ and $T(u'
			 \oplus r_1, \lceil A(x) \rceil, \rho*c)$.
			 Since $T(u', \lceil t \leq s \rceil, \rho)$,
			 $\exists d. v(\lceil s \rceil,
			 \rho)\downarrow_u' d$ and $c \leq d$.
			 Therefore, by definition of $T$, $T(u' \oplus
			 r_1, \lceil \exists x \leq s. A(x) \rceil,
			 \rho)$.

	     \end{description}
	      \item[Cut-rule] $$\infer{\Gamma, \Pi \rightarrow \Delta,
			 \Lambda}{\infer*[r_1]{\Gamma \rightarrow
	    \Delta, A}{} \quad
	                          \infer*[r_2]{
			 A, \Pi \rightarrow \Lambda}{}}$$

	    Assume $\forall C \in \Gamma, \Pi, T(u', \lceil C \rceil,
	    \rho)$.  By induction hypothesis on $r_1$, either $\exists D
	    \in \Delta, T(u' \oplus r_1, \lceil C \rceil, \rho)$ or
	    $T(u' \oplus r_1, \lceil A \rceil, \rho)$.  If $\exists D
	    \in \Delta, T(u' \oplus r_1, \lceil C \rceil, \rho)$ then
	    $T(u' \oplus r, \lceil C \rceil, \rho)$ by Lemma
	    \ref{lem:T}, therefore we have done.  Hence, we assume that
	    $T(u' \oplus r_1, \lceil A \rceil, \rho)$.  By assumption
	    and Lemma \ref{lem:T}, $\forall C \in \Pi, T(u' \oplus r_1,
	    \lceil C \rceil, \rho)$.  Since $u' \oplus r_1 \leq u
	    \ominus r_2$, we can apply induction hypothesis to $u' \oplus
	    r_1$.  Hence, we have $\exists D \in \Lambda, T(u' \oplus
	    r_1 \oplus r_2, \lceil D \rceil, \rho)$.  Since $u' \oplus
	    r_1 \oplus r_2 \leq u' \oplus r$, $\exists D \in \Lambda,
	    T(u' \oplus r, \lceil D \rceil, \rho)$ by Lemma
	    \ref{lem:T}.  Hence we have done.
\end{description}
\end{proof}

\begin{theorem}
 $S^2_2 \vdash \forall w \neg 1-Prf(w, \lceil \rightarrow \rceil)$
\end{theorem}

\begin{proof}
 Immediate from Proposition \ref{prop:soundness}.
\end{proof}

\section{Conjectures}\label{sec:conjectures}

In this section, we discuss several conjectures concerning $S^0_2 E$.

The most interesting problem concerning $S^0_2 E$ is whether $S^2_1$
 proves $\forall w \neg 1-Prf(w, \lceil \rightarrow \rceil)$ or not.  If
 the answer is negative, we have $S^2_1 \not= S^2_2$, hence the
 fundamental problem of bounded arithmetic is solved.

It would be easier to prove 

 $$S^2_1 \not\vdash \forall w \neg Prf(w, \lceil \rightarrow \rceil)$$

allowing any formula in the proof, since Solovay's cut shortening
technique would work.  To use Solovay's cut shortening technique, we
need to convert $S^2_1$ proof to $S^0_2 E$ proof with $\Sigma^b_1$-PIND.
This is achieved another conjecture.  Let $\phi(\vec{x})$ be a
$\Sigma^b_1$-formula with free variables $\vec{x} = x_1, x_2, \dots,
x_n$.  Assume $S^2_1 \vdash \phi(\vec{x})$.  Then, we conjecture $S^2_0
E + \Sigma^b_1-PIND \vdash N\vec{x} \rightarrow \phi(\vec{x})$, where
$N\vec{x}$ stands for the sequent $Nx_1, Nx_2, \dots, Nx_n$.  This is
plausible because $S^2_0E$ contains all inductive definition necessary
to prove totality and uniqueness of functions and predicates.

\bibliography{../my.bib,../yoriyuki-bounded-arithmetic.bib}

\end{document}